\documentclass[12pt]{article}
\usepackage{latexsym}
\usepackage{epsfig}
\usepackage{theorem}
\usepackage{amsmath}
\usepackage{amssymb}
\usepackage{graphics}
\usepackage{epsfig}

\newenvironment{proof}{
\noindent {\bf Proof.\hskip 3mm}}%
{\mbox{}\hfill\rule{0.5em}{0.809em}\par}
\newtheorem{theorem}{{\bf Theorem}}[section]
\newtheorem{lemma}{{\bf Lemma}}[section]
\newtheorem{corollary}[theorem]{{\bf Corollary}}

\newtheorem{conjecture}{{\bf Conjecture }}[section]
\newtheorem{problem}[conjecture]{{\bf Problem}}

\newtheorem{fact}{{\bf Fact}}[section]
\newenvironment{pf1}{\noindent {\bf Proof of Theorem \ref{mad<4}.}}{\hfill\rule{2mm}{3mm}\par\medskip}
\newenvironment{pf2}{\noindent {\bf Proof of Theorem \ref{mad<3}.}}{\hfill\rule{2mm}{3mm}\par\medskip}
\newcommand\mad{\mathrm{mad}}
\newcommand\path{maximal bichromatic path }

\begin{document}
\setcounter{page}{1}

\title{Acyclic edge coloring of sparse graphs
\thanks{ This work is  supported by  research grants NSFC (11001055) and NSFFP (2010J05004, 2011J06001). }}
\author{ Jianfeng Hou
\\Center for Discrete Mathematics, Fuzhou University, \\Fujian, P.
R. China, 350002\\ jfhou@fzu.edu.cn
}

\date{}
\maketitle \vspace{-0.2cm}
\begin{abstract}
A proper edge coloring of a graph $G$ is called acyclic if there is
no bichromatic cycle in $G$. The acyclic chromatic index of $G$,
denoted by $\chi'_a(G)$, is the least number of colors $k$ such that
$G$ has an acyclic edge $k$-coloring. The maximum average degree of
a graph $G$, denoted by $\mad(G)$, is the maximum of the average
degree of all subgraphs of $G$.  In this paper, it is  proved that
if $\mad(G)<4$, then $\chi'_a(G)\leq{\Delta(G)+2}$; if $\mad(G)<3$,
then $\chi'_a(G)\leq{\Delta(G)+1}$. This implies that every
triangle-free planar graph $G$ is acyclically edge
$(\Delta(G)+2)$-colorable.
\end{abstract}


\vspace{5mm}

\section{Introduction}

In this paper, all considered graphs are finite, simple and
undirected. We use $V(G)$, $E(G)$, $\delta(G)$ and $\Delta(G)$ (or
$V, E, \delta$ and $\Delta$ for simple) to denote the vertex set,
the edge set, the minimum degree and the maximum degree of a graph
$G$, respectively. For a vertex $v\in V(G)$, let $N(v)$ denote the
set of vertices adjacent to $v$ and $d(v)=|N(v)|$ denote the
$degree$ of $v$. Let $N_k(v)=\{x\in N(v)|d(x)=k\}$ and
$n_k(v)=|N_k(v)|$. A vertex of degree $k$ is called a {\em $k$-vertex}. We
write a {\em $k^+$-vertex} for a vertex of degree at least $k$, and
a {\em $k^-$-vertex} for that of degree at most $k$. The girth of a
graph $G$, denoted by $g(G)$, is the length of its shortest cycle.



As usual $[k]$ stands for the set $\{1,2,\ldots,k\}$.

A $proper$ $edge$ $k$-$coloring$ of a graph $G$ is a mapping $\phi$
from  $E(G)$ to the color set $[k]$ such that no pair of adjacent
edges are colored with the same color.  A proper  edge coloring of a
graph $G$ is called {\em acyclic} if there is no bichromatic cycle
in $G$. In other words, if the union of any two color classes
induces a subgraph of $G$ which is a forest.  The {\em acyclic
chromatic index} of $G$, denoted by $\chi'_a(G)$, is the least
number of colors $k$ such that $G$ has an acyclic edge $k$-coloring.


Acyclic edge coloring has been widely studied over past twenty years. The first general linear upper bound on $\chi'_a(G)$ was found by Alon et al. in \cite{Alon1991}. Namely, they proved that $\chi'_a(G)\leq{64\Delta(G)}$. This bound was improved to $16\Delta(G)$ by
Molloy and Reed \cite{molloy1998} using the  same method. In 2001, Alon,
Sudakov and Zaks \cite{Alon2001} stated the Acyclic Edge
Coloring Conjecture as follows.

\begin{conjecture} \label{conjecture}
$\chi'_a(G)\leq{\Delta(G)+2}$ for
all graphs $G$.
\end{conjecture}

Obviously, $\chi'_a(G)\le \Delta(G)+1$ for graphs $G$ with
$\Delta(G)\le 2$. Andersen et al. \cite{and2012} showed that if $G$
is a connected graph with $\Delta(G)\le 3$ different from $K_4$ and
$K_{3,3}$, then $\chi'_a(G)\leq 4$. Hence Conjecture
\ref{conjecture} holds for $\Delta(G)\le 3$. This conjecture was
also verified for some special classes of graphs, including
non-regular graphs with maximum degree at most four \cite{basa2009},
outerplanar graphs \cite{hou2009, mns07a}, series-parallel graphs
\cite{hou2008}, planar graphs with girth at least five \cite{boro2009, hou2008},
graphs with large girth \cite{Alon2001}, and so on.


Recall that the $maximum$ $average$ $degree$ of a graph $G$, denoted
by $\mad(G)$, is the maximum of the average degree of all of its
subgraphs, i.e., $\mad(G)=\max_{H\subseteq
G}\frac{2|E(H)|}{|V(H)|}$. Recently, Basavaraju and Chandran
\cite{basa2012} consider the acyclic edge coloring of sparse graphs
and proved that if $\mad(G)<4$, then $\chi'_a(G)\leq \Delta(G)+3$.
In this paper, the bound is improved  to $\Delta(G)+2$.

\begin{theorem} \label{mad<4}
Let $G$ be a  graph with $\mad(G)<4$. Then $\chi'_a(G)\leq \Delta(G)+2$.
\end{theorem}

\begin{theorem} \label{mad<3}
Let $G$ be a  graph with $\mad(G)<3$. Then $\chi'_a(G)\leq \Delta(G)+1$.
\end{theorem}

By an  application of Euler's formula, it is easy to see that
every planar graph $G$ with girth $g(G)$ satisfies
$\mad(G)<\frac{2g(G)}{g(G)-2}$. Thus, we have the following corollaries.

\begin{corollary} \label{corollary1}
Every  triangle-free planar graph is acyclic edge $(\Delta(G)+2)$-colorable.
\end{corollary}

\begin{corollary} \label{corollary2}
Every  planar graph with $g(G)\ge 6$ is acyclic edge $(\Delta(G)+1)$-colorable.
\end{corollary}


\section{Properties of $k$-critical graph}

At first, we introduce some notations and definitions used in \cite{basaSIAM}.
Let $H$ be a subgraph of $G$. Then an acyclic  edge coloring $\phi$
of $H$ is called a  $partial$ $acyclic$ $edge$ $coloring$ of $G$.
Note that $H$ can be $G$ itself. Let $\phi: E(G)\rightarrow [k]$ be a partial  edge
$k$-coloring of $G$. For any vertex $v\in V(G)$, let
$F_v(\phi)=\{\phi(uv)|u\in{N_G(v)}\}$ and $C_v(\phi)=[k]-F_v(\phi)$.
For an edge $uv\in E(G)$, we say the color $\phi(uv)$ $appears$ on
the vertex $v$ and define $S_{uv}(\phi)=F_v(\phi)-\{\phi(uv)\}$.
Note that $S_{uv}(\phi)$ need not be the same as $S_{vu}(\phi)$.

Let $\alpha,\beta$ be two colors.
An $(\alpha,\beta)$-$maximal$ $bichromatic$ $path$ with respect to a partial edge coloring of $G$ is maximal path  consisting of edges that are colored $\alpha$ and $\beta$ alternatingly.  An $(\alpha,\beta,u,v)$-$maximal$ $bichromatic$ $path$ is an $(\alpha,\beta)$-maximal bichromatic path which starts at the vertex $u$ with an edge colored $\alpha$ and ends at $v$. An $(\alpha,\beta,u,v)$-maximal bichromatic path which ends at the vertex $v$ via an edge colored $\alpha$ is called an $(\alpha,\beta,u,v)$-$critical$ $path$.

A graph $G$ is called an $acyclically$ $edge$ $k$-$critical$ $graph$
if $\chi'_a(G)>k$ and any proper subgraph of $G$ is acyclically edge
$k$-colorable. Obviously, if $G$ is an acyclically edge $k$-critical
graph with $k>\Delta(G)$, then $\Delta(G)\ge 3$. The following facts are obvious.


\begin{fact} \label{one-bichromatic-path}
Given a pair of color $\alpha$ and $\beta$ of a proper edge coloring $\phi$ of $G$, there is at most one  $(\alpha,\beta)$-maximal bichromatic path containing a particular vertex $v$, with respect to $\phi$.
\end{fact}

\begin{fact} \label{uv-edge-degree}
Let $G$ be an acyclically edge $k$-critical graph, and $uv$ be an edge of $G$. Then for any acyclically edge $k$-coloring $\phi$ of $G-uv$, if $F_u(\phi)\cap F_v(\phi)=\emptyset$, then $d(u)+d(v)\ge k+2$. If $|F_u(\phi)\cap F_v(\phi)|=t$, say $\phi(uu_i)=\phi(vv_i)$ for $i=1,2,..,t$, then $\sum\limits_{i=1}^{t}d(v_i)+d(u)+d(v)\ge k+t+2$ and $\sum\limits_{i=1}^{t}d(u_i)+d(u)+d(v)\ge k+t+2$.
\end{fact}

In \cite{hou2011}, Hou et al.  considered the properties of
acyclically edge $k$-critical graphs and got the following lemmas.

\begin{lemma} \cite{hou2011}\label{2-connected}
Any acyclically edge $k$-critical graph  is $2$-connected.
\end{lemma}

\begin{lemma} \cite{hou2011}\label{lem2vertex}
Let $G$ be an acyclically edge $k$-critical graph with $k\le 2\Delta(G)-2$ and $v$ be a vertex of
$G$  adjacent to a $2$-vertex. Then $v$ is
adjacent to at least $k-\Delta(G)+1$ vertices of degree at least $k-\Delta(G)+2$.
\end{lemma}

By above lemma, we have the following corollaries.

\begin{corollary}\label{2-vertex-+1}
Let $v$ be a vertex of an acyclically edge $(\Delta(G)+1)$-critical
graph $G$. Then $n_2(v)\leq \Delta(G)-2$.
\end{corollary}

\begin{corollary}\label{2and3-vertex-+2}
Let $v$ be a vertex of an acyclically edge $(\Delta(G)+2)$-critical graph  $G$. If $n_2(v)\neq 0$, then $n_2(v)+n_3(v)\leq \Delta(G)-3$.
\end{corollary}

Now we consider the neighbors of 2-vertices in  acyclically edge critical graphs.

\begin{lemma}  \label{neighbor of 2-vertex}
Let $v$ be a $2$-vertex of   an acyclically edge $k$-critical graph $G$ with
$k>\Delta(G)$. Then the neighbors of
$v$ are $(k-\Delta(G)+3)^+$-vertices.
\end{lemma}

\begin{proof} Suppose to the contrary that $v$ has a neighbor $u$
whose degree is at most $k-\Delta(G)+2$. Let $N(v)=\{u,w\}$ and
$N(u)=\{v, u_1,u_2,\ldots,u_t\}$, where $t\le k-\Delta(G)+1$. Then
the graph $G'=G-uv$ admits an acyclic edge $k$-coloring $\phi$ by
the choice of $G$ with $\phi(uu_i)=i$ for $1\le i\le t$. Since
$d(u)+d(v)\le \Delta(G)+2$, we have $\phi(wv)\in F_u(\phi)$ by Fact
\ref{uv-edge-degree}, say $\phi(wv)=1$.   Then for any $t+1\le i\le
k$, there is a $(1,i,u,v)$-critical path with respect to $\phi$
through $u_1$ and $w$, since otherwise we can color $uv$ with $i$
properly with avoiding bichromatic cycle, which is a contradiction
to the choice of $G$. Thus $d(u_1)\ge k-t+1$. This implies that
$t=k-\Delta(G)+1$ and $C_{u_1}(\phi)=C_{w}(\phi)=\{2,3,\ldots,t\}$.
Recolor $wv$ with 2, by the same argument, there is a
$(2,i,u,v)$-critical path  through $u_2$ and $w$ for any $t+1\le
i\le k$. Now exchange the colors on $uu_1$ and $uu_2$, and color
$uv$ with $t+1$. The resulting coloring is an acyclic edge coloring
of $G$ using $k$ colors by Fact \ref{one-bichromatic-path}, a
contradiction.
\end{proof}

In \cite{hou2011}, Hou et al. got the following lemma.

\begin{lemma} \cite{hou2011}\label{lem3vertex}
Let $G$ be an acyclically edge $k$-critical graph with $k\ge
\Delta(G)+2$ and $v$ be a $3$-vertex of $G$. Then the neighbors of
$v$ are $(k-\Delta(G)+2)^+$-vertices.
\end{lemma}

Lemma \ref{lem3vertex} implies that the neighbors of a $3$-vertex
are $4^+$-vertices in an acyclically edge $(\Delta(G)+2)$-critical
graph. Now we give the following lemma.

\begin{lemma}  \label{3-vertex-adjacent-4-vertex}
Let $v$
be a $3$-vertex of   an acyclically edge $(\Delta(G)+2)$-critical graph $G$ with neighbors $x,y,z$. If $d(x)=4$, then

\begin{enumerate}
\item both  $y$ and $z$ are $5^+$-vertices.
\item one of vertices in $\{y,z\}$, say $y$, is  adjacent to at least three $4^+$-vertices. The other vertex $z$ is adjacent to at least two $4^+$-vertices. Moreover, if $d(z)= 5$, then $z$ is adjacent to at least three $4^+$-vertices.
\end{enumerate}
\end{lemma}

\begin{proof} We may assume that $N(x)=\{v,x_1,x_2,x_3\}$. Let $k=\Delta(G)+2$ for
simplicity. Then for any cyclic edge $k$-coloring $\phi$ of
$G'=G-vx$,  $F_v(\phi)\cap F_x(\phi)\neq \emptyset$ by Fact
\ref{uv-edge-degree}.

Now we show that there exists an cyclic edge $k$-coloring $\phi$ of
$G'=G-vx$, such that $|F_v(\phi)\cap F_x(\phi)|=1$. Otherwise, for
any acyclic edge coloring $\varphi$ of $G'$,  $|F_v(\varphi)\cap
F_x(\varphi)|=2$. Assume that  $\varphi(xx_i)=i$ for $1\le i\le 3$,
$\varphi(vy)=1$ and $\varphi(vz)=2$. Let $S$ (or $T$) denote the set
of colors such that for any $i\geq 4$ and $i\in S$ (or $i\in T$),
there is  a $(1,i,v,x)$-critical path (or $(2,i,v,x)$-critical path)
through $x_1$ and $y$ (or $x_2$ and $z$). Then $S\cup
T=\{4,5,\ldots,k\}$ by the choice of $G$. Assume that $T\neq
\emptyset$.

If there exists a  color $\alpha\in T$ such that $\alpha$ does not appear on $y$, then recolor $vy$ with $\alpha$ to get a partial edge-coloring $\phi$, i.e., $\phi(vy)=\alpha$, and $\phi(e)=\varphi(e)$ for all other edges $e$ in $G$.
Then $\phi$ is an acyclic edge coloring of $G'$ with $|F_v(\phi)\cap F_x(\phi)|=1$ by Fact \ref{one-bichromatic-path}. Otherwise, $T\subseteq F_y(\varphi)$. Similarly,  $S\subseteq F_z(\varphi)$. This implies that $d(y)=d(z)=\Delta(G)$ and  $S\cap T=\emptyset$.

Exchange the colors on $xy$ and $xz$. Then for any $i\in S$, there is a $(2,i,v,x)$-critical path from through $x_2$ and $y$; for any $j\in T$, there is a $(1,j,v,x)$-critical path through $x_1$ and $z$. This implies that $S_{xx_i}(\varphi)=S\cup T$ for $i=1,2$. If recolor $vy$ with 3, then for any $i\in S$, there is a $(3,i,v,x)$-critical path through $x_3$ and $vy$. If exchange the colors on $xx_1$ and $xx_2$, then for any $i\in T$, there is a $(3,i,v,x)$-critical path through $x_3$ and $y$. We recolor $vz$ with 3 and choose a color from $T$ to color $vx$. It is easy to verify that the resulting coloring is an acyclic edge coloring of $G$, a contradiction.

Now suppose that $\phi$ is an acyclic edge $k$-coloring  of
$G'=G-vx$ with $|F_v(\phi)\cap F_x(\phi)|=1$. Let $\phi(xx_i)=i$ for
$1\le i\le 3$, $\phi(vy)=1$ and $\phi(vz)=4$. Then for any $5\le
i\le k$, there is a $(1,i,v,x)$-critical path through $x_1$.

{\bf Claim 1.} $4\in F_y(\phi)$

\begin{proof}
By contradiction, $4\notin F_y(\phi)$.  Since $|C_y(\phi)|\ge 2$,
$\{2,3\}\cap C_y(\phi)\neq \emptyset$, say $2\in C_y(\phi)$. Recolor
$vy$ with 2, the resulting coloring is also a acyclic edge coloring
of $G'$. Then there is a $(2,i,v,x)$-critical path  through $x_2$
for any $i\ge 5$. Note that $\{1,2\}\subseteq  F_z(\phi)$, since
otherwise we can recolor $vy$ with 4, $vz$ with 1 or 2, and color
$xv$ with 5. This implies that there is a color $i_0$ with $5\le
i_0\le k$ such that $i_0$ does not appear on $z$. Recolor $vz$ with
$i_0$ and color $xv$ with 4.\end{proof}

It follows from Claim 1 that $d(y)=\Delta(G)$ and
$S_{vy}(\phi)=\{4,5,\ldots,k\}$.

{\bf Claim 2.} $1\in F_z(\phi)$.

\begin{proof} Otherwise, let $\phi'(vy)=2$, $\phi'(vz)=1$, and $\phi'(e)=\phi(e)$ for the other edges $e$ of $G'$. Then $\phi'$ is also an acyclic edge coloring of $G'$. Note that coloring $vx$ with $i$ does not create bichromatic cycle containing the edges $vz$ and $xx_1$ by Fact \ref{one-bichromatic-path} for any $4\le i\le k$. Thus, there is a $(2,i,v,x)$-critical path for any $4\le i\le k$, since otherwise, we can color $vx$ with $i$. Similarly, let  $\phi'(vy)=3$, $\phi'(vz)=1$, and $\phi'(e)=\phi(e)$ for the edges $e$. Then  there is a $(3,i,v,x)$-critical path for any $4\le i\le k$ with respect to $\phi'$. Now we give an acyclic edge coloring $\varphi$ of $G$ by letting $\varphi(xx_2)=\phi(xx_3)$, $\varphi(xx_3)=\phi(xx_2)$, $\varphi(vy)=2$, $\varphi(vx)=5$, and $\varphi(e)=\phi(e)$ for the other edges $e$ of $G$, a contradiction.\end{proof}

{\bf Claim 3.} There is a $(4,2,v,y)$-critical path  through $z$ with respect to $\phi$.

\begin{proof}
Otherwise, recoloring $vy$ with 2 results in a new acyclic edge
coloring $\phi'$ of $G'$. Then there is a $(2,i,v,x)$-critical path
through $x_2$ for any $i\in\{5,\ldots,k\}$ with respect to $\phi'$.
This implies that  $C_{x_2}(\phi)\subseteq\{1,3,4\}$.

Now we show that there is a $(1,4,v,x)$-critical path through $x_1$
and $y$ with respect to $\phi$. Otherwise, coloring $vx$ with 4 and
uncoloring $vz$ results in an acyclic edge coloring $\varphi$ of
$G''=G-vz$. If $2\in C_{y}(\varphi)$, then color $vz$ with 2.
Otherwise, there exists a color $i_0\in \{5,6,\ldots,k\}$ such that
$i_0\in C_{y}(\varphi)$. Color $vz$ with $i_0$ and get an acyclic
edge coloring of $G$.

By the same argument above, there is a $(2,4,v,x)$-critical path through $x_2$ and $y$ if we recolor $vy$ with 2. Let $\varphi(xx_1)=2$, $\varphi(xx_2)=1$,  $\varphi(vx)=5$, and $\varphi(e)=\phi(e)$ for the other edges $e$ of $G$. Then $\varphi$ is an acyclic edge coloring of $G$, a contradiction. \end{proof}




As in the proof of Claim 3, there is a $(3,4,v,y)$-critical path
through $z$. Thus $\{1,2,3,4\}\subseteq F_z({\phi})$ by Claims 2 and
3. This implies that at least two colors in $\{5,6,\ldots,k\}$ such
that neither appear on $z$, say 5, 6. Note that there is a
$(1,4,v,x)$-critical path  through $x_1$ and $y$, since otherwise,
recolor $vz$ with 2 and color $vx$ with 4. Recolor $vz$ with $5$ (or
6), the resulting coloring $\phi'$ is also an acyclic edge coloring
of $G'$. As in the proof of Claim 2, there is a $(5,i, v,
y)$-critical path (or $(6,i, v,y)$-critical path)  through $z$ with
respect to $\phi'$ for any $i\in \{2,3\}$. Let $\phi(yy_i)=i$ for
$i=4,5,6$ and $\phi(zz_i)=i$ for $i=1,2,3$. Then
$\{1,2,3,i\}\subseteq F_{y_i}(\phi)$ for $i=4,5,6$ and
$\{4,5,6,i\}\subseteq F_{z_i}(\phi)$ for $i=2,3$. Thus $y$ is
adjacent to at least three $4^+$-vertices and $z$ is adjacent to at
least two $4^+$-vertices.

{\bf Claim 4.} $d(y)\ge 5$

\begin{proof}
Otherwise, $\Delta(G)=4$ and $F_{y_i}(\phi)=\{1,2,3,i\}$ for any
$i=4,5,6$. Exchange the colors on $yy_4$ and $yy_5$, recolor $vz$
with 6, and color $vx$  with 5. It is easy to see that the resulting
coloring is an acyclic edge coloring of $G$, a contradiction.
\end{proof}

{\bf Claim 5.} $d(z)\ge 5$

\begin{proof}
Otherwise, $d(z)=4$ and $F_{z}(\phi)=\{1,2,3,4\}$. As in the proof
of Claim 3, there is a $(2,i,v,z)$-critical path (or
$(3,i,v,z)$-critical path) for any $5\le i\le k$ if recoloring $vy$
with 2 (or 3).   This implies $F_{z_i}(\phi)=\{i, 4,5,\ldots,k\}$.
Exchange the colors on $zz_1$ and $zz_2$ and the argument in Claim 3
works.
\end{proof}

Finally, we prove that if $d(z)=5$, then $n_3(z)\leq 2$. Assume that
$N(z)=\{z_1,z_2,z_3,z_7,v\}$ and $F_z(\phi)=\{1,2,3,4,7\}$. Suppose
that  $n_3(v)\ge 3$, then $d(z_1)=d(z_7)=3$ by Lemma
\ref{lem2vertex}, say $N(z_i)=\{v,z'_i,z''_i\}$ for $i=1,7$. Recolor
$vy$ with 2, $vz$ with 1, color $vx$ with 4, and uncolor $zz_1$.
Then the resulting coloring is acyclic edge coloring $\varphi$ of
$G''=G-zz_1$ by Fact \ref{one-bichromatic-path}. By the choice of $G$, we have $F_{z_1}(\varphi)\cap
F_z(\varphi)\neq \emptyset$.

We first consider the case that $|F_{z_1}(\varphi)\cap
F_z(\varphi)|=1$. If $7\notin F_{z_1}(\varphi)$, then coloring
$zz_1$ with a color from $\{4,5,6\}\setminus F_{z_1}(\varphi)$.
Otherwise, assume that $\varphi(z_1z'_1)=4$ and
$\varphi(z_1z''_1)=7$. Then for any $5\le i\le k$ and $i\neq
7$, there is a $(7,i,z,z_1)$-critical path through $z_7$. This
implies that  $\Delta(G)=5$ and $F_{z_{7}}(\varphi)=\{5,6,7\}$. Note
that $7\notin F_{z_{2}}(\varphi)$, since otherwise recolor $vz$ with
5, $zz_2$ with 1, $zz_7$ with 4, and color $zz_1$ with 2. Similarly,
$7\notin F_{z_{3}}(\varphi)$. Then recolor $zz_2$ with 7, $zz_7$
with 4, and color $zz_1$ with 2.

Now we consider the case that $|F_{z_1}(\varphi)\cap
F_z(\varphi)|=2$. If $7\notin F_{z_{1}}(\varphi)$, then color $zz_1$
with 5. Otherwise, color $zz_1$ with a color from
$\{4,5,6\}\setminus F_{z_{7}}(\varphi)$.

In any case, we can get an acyclic edge coloring of $G$, a
contradiction. \end{proof}

In \cite{basa2009},  Basavaraju and Chandran showed that if a connected graph $G$ with $n$ vertices and $m$ edges satisfies $m\le 2n-1$ and maximum degree $\Delta(G)\le 4$, then $\chi'_a(G)\le 6$. This result can be obtained immediately by Lemmas \ref{neighbor of 2-vertex} and \ref{3-vertex-adjacent-4-vertex}.

\begin{lemma}  \label{k-vertex-2-ver}
Let $v$
be a $t$-vertex of  an acyclically edge $(\Delta(G)+2)$-critical graph $G$ with $t\ge 5$. Then $n_2(v)\leq t-4$. Moreover, if $n_2(v)= t-4$, then $n_3(v)=0$.
\end{lemma}

\begin{proof}
Suppose that $v$ has neighbors $x_1,x_2,\ldots,x_t$ with
$N(x_i)=\{v, y_i\}$ for $i=5,6,\ldots,t$ and $k=\Delta(G)+2$ for
simplicity. We only have to prove that $d(x_i)\ge 4$ for $1\le i \le
4$. By contradiction, assume that $d(x_4)\le 3$. We only consider
the case that $d(x_4)=3$, say $N(x_4)=\{y'_4,y''_4\}$. The case that
$d(x_4)=2$ is the same, but easier. The graph $G'=G-vx_t$ has an
acyclic edge $k$-coloring $\phi$  with $\phi(vx_i)=i$ for $1\le i\le
t-1$. Since $d(v)+d(x_i)\leq \Delta(G)+3$ for $i\ge 4$, we have
$\phi(x_ty_t)\in \{1,2,3\}$ by Fact \ref{uv-edge-degree}. Without
loss of generality, let $\phi(x_ty_t)=1$. Note that
$d(y_t)=\Delta(G)$  and $C_{y_t}(\phi)=\{2,3\}$, since otherwise,
recolor $x_ty_t$ with a color $i\ge 4$  not appearing on $y_t$, and
color $x_ty_t$ with a coloring  appearing neither on $v$ nor on
$x_i$. Moreover, for any $t\le i\le k$, there is a $(1,i, v,
x_t)$-critical path through $x_1$. Similarly, if recolor $x_ty_t$
with 2 (or 3), then there is a $(2,t, v, x_t)$-critical path (or
$(2,i, v, x_t)$-critical path) through $x_2$ (or $x_3$) for any
$t\le i\le k$.

{\bf Claim 1.} For any color $i$ with $4\le i\le t-1$, there is a $(1,i,v,x_t)$-critical path  through $x_1$ and $y_t$ with respect to $\phi$.

\begin{proof} Otherwise, there exists a color $i_0\in \{4,5,\ldots,t-1\}$ such that there is no $(1,i_0,v,x_t)$-critical path  through $x_1$ and $y_t$. Recolor $vx_t$ with $i_0$, uncolor $vx_{i_0}$ and the colors on the other edges unchange. We get an acyclic edge coloring $\varphi$ of $G''=G-vx_{i_0}$.

We first consider the case that $i_0\ge 5$, say $i_0=5$. Since $d(v)+d(x_5)\leq \Delta(G)+2$, we have $\phi(vx_t)\in \{1,2,3\}$ by Fact \ref{uv-edge-degree}. Color $vx_5$ with $t$ and get an acyclic edge coloring of $G$ by Fact \ref{one-bichromatic-path}, a contradiction.

Now we consider the case that $i_0=4$. Assume that $F_{x_4}(\varphi)=\{\alpha,\beta\}$. If $F_{x_4}(\varphi)\cap \{1,2,3\}= \emptyset$, we are done by Fact \ref{one-bichromatic-path}. Otherwise, let $\alpha=1$. For any $i\ge t$, since there is a $(1,i,v,x_t)$-critical path, there is no $(1,i,v,x_4)$-critical path. If $\beta\in \{2,3\}$, then color $vx_4$ with $t$. If $5\le \beta\le t-1$, then color $vx_4$ with a color appearing neither  on $v$ nor on $x_{\beta}$. Otherwise, color $vx_4$ with a color appearing on neither  $v$ nor on $x_{4}$.
\end{proof}

Recolor $x_ty_t$ with 2 (or 3), as in the proof of Claim 1, there is
a $(2,i,v,x_t)$-critical path (or $(3,i,v,x_t)$-critical path)
through $y_t$ and $x_2$ or ($x_3$). Thus
$S_{xx_i}(\phi)=\{4,5,\ldots,k\}$ for $i=1,2,3$.  For $1\le i<j\le
3$, let $\phi_{ij}$ denote the coloring that exchanging the colors
on $vx_i$ and $vx_j$, and coloring $vx_t$ with $t$ with respect to
$\phi$.

{\bf Claim 2.} There is a coloring in $\{\phi_{ij}: 1\le i<j\le 3\}$ such that there is no bichromatic cycle containing $vv_4$.

\begin{proof}
If $|F_{v_4}(\phi)\cap \{1,2,3\}|\le 1$, then there exist two colors
$i_0,j_0$ with $1\le i_0<j_0\le 3$ such that $i_0,j_0$ neither
appear on $v_4$.  The coloring $\phi_{i_0j_0}$ satisfies the
condition. Otherwise, assume that $F_{v_4}(\phi)=\{1,2,4\}$. If the
coloring $\phi_{12}$ has no bichromatic cycle containing $vv_4$, we
are done. Otherwise,  there is either a $(1,4,v, x_2)$-critical
through $x_4$, or  a $(2,4,v,x_1)$-critical path through $x_4$ with
respect to $\phi_{12}$. In this case,  the coloring $\phi_{13}$ has
no bichromatic cycle containing $vv_4$ by Fact
\ref{one-bichromatic-path}.
\end{proof}

We may assume the coloring $\varphi=\phi_{12}$ is the coloring
containing no bichromatic cycle containing $vv_4$. Let $S=\{x_i:5\le
i\le t-1$, there is a $(i,1,v, x_2)$-critical path through $x_i$
with respect to $\varphi\}$, and $T=\{x_j:5\le j\le t-1$, there is a
$(j,2,v,x_1)$-critical path through $x_j$ with respect to
$\varphi\}$. Then $S\cup T\neq \emptyset$ by the choice of $G$.  We
may assume that $|S|\le1$, since otherwise, let
$S=\{x_{i_1},x_{i_2},\ldots,x_{i_t}\}$,  where $t\ge 2$. We recolor
$vx_{i_p}$ with $\varphi(vx_{i_{p+1}})$ for $p=1,2,\ldots,t-1$,
$vx_{i_t}$ with $\varphi(vx_{i_{1}})$. Then the resulting coloring
has no  bichromatic cycle containing $vx_2$. Similarly,  assume that
$|T|\le 1$. We consider the following three cases.

{\bf Case 1.} $|S|=1$ and $|T|=0$.

Assume that $S=\{5\}$. Then $\varphi(x_5y_5)=1$. We consider the colors not appearing on $y_5$. If there is a color $i\geq t$ such that $i\in C_{y_5}(\varphi)$, then recolor $x_5y_5$ with $i$.  If there is a color $6\le i\le t-1$ such that $i\in C_{y_5}(\varphi)$, then recolor $x_5y_5$ with $i$, $vx_5$ with a color appearing neither   on $v$ nor on $x_i$. If $3\in C_{y_5}(\varphi)$, then recolor $x_5y_5$ with 3, color $vx_5$ with $t+1$. Otherwise,  $4\in C_{y_5}(\varphi)$. Then recolor $x_5y_5$ with 4 and get a new coloring $\varphi'$ of $G$. If $\varphi'$ is  acyclic, we are done. Otherwise, recolor $vx_5$ with a color appearing neither on $v$ nor $x_4$ in $\varphi'$.

{\bf Case 2.} $|T|=1$ and $|S|=0$.

This case is the same as  Case 1.

{\bf Case 3.} $|S|=|T|=1$.

Assume that $S=\{5\}$ and $T=\{6\}$. Then $\varphi(x_5y_5)=1$ and
$\varphi(x_6y_6)=2$. If we exchange the colors on $vx_5$ and $vx_6$,
then there is a $(1,6,v,x_5)$-critical path  through $y_5$ and a
$(2,5, v, x_6)$-critical path through $y_6$, since otherwise, the
argument in Case 1 and Case 2 works. If there exists a color
$\alpha$ with $\alpha\ge t$ or $\alpha=1$, such that $\alpha\in
C_{y_6}(\varphi)$, then recolor $x_6y_6$ with $\alpha$ and Case 1
works. If $3\in C_{y_6}(\varphi)$, we recolor $x_6y_6$ with 3,
exchange the colors $vx_6$ and $vx_t$, and then Case 1 works.
Otherwise, there is a color $5\le j_0\le t-1$ such that $j_0\in
C_{y_6}(\varphi)$. Recolor $x_6y_6$ with $j_0$ in $\varphi$ and the
colors on other edges unchange. Then we get a new coloring
$\varphi'$ of $G$. In $\varphi'$, if there is no bichromatic cycle
containing $vx_6$, then Case 1 works. Otherwise, recoloring $vx_6$
with a color appearing neither on $v$ nor on $x_{j_0}$, and then
Case 1 also works.
\end{proof}

\begin{lemma}  \label{5-vertex}
Let $v$
be a $5$-vertex of  an acyclically edge $(\Delta(G)+2)$-critical graph $G$. Then $n_2(v)+n_3(v)\leq 3$.
\end{lemma}

\begin{proof}
If $n_2(v)>0$, we are done by Corollary \ref{2and3-vertex-+2}.
Suppose that $n_2(v)=0$ and $k=\Delta(G)+2$ for simplicity. By
contradiction, let $v$ be 5-vertex with neighbors
$v_1,v_2,\ldots,v_5$ such that $N(v_i)=\{v,x_i,y_i\}$ for
$i=2,3,4,5$. The graph $G'=G-vv_5$ has an acyclic  edge $k$-coloring
$\phi$ with $\phi(vv_i)=i$ for $i=1,2,3,4$. Then $F_v(\phi)\cap
F_{v_5}(\phi)\neq \emptyset$ by Fact \ref{uv-edge-degree}.

{\bf Case 1.} $|F_v(\phi)\cap F_{v_5}(\phi)|=1$.

Assume that $\phi(v_5y_5)=5$. We consider the color on $v_5x_5$.

{\bf Subcase 1.1.} $\phi(v_5x_5)\in \{2,3,4\}$, say $\phi(v_5x_5)=2$.

For any $i\ge 6$, there is a $(2,i,v,v_5)$-critical path through
$v_2$ and $x_5$, since otherwise, we can color $vv_5$ with $i$. This
implies that $\Delta(G)=5$ and $F_{v_2}(\phi)=\{2,6,7\}$. This is also
a $(5,i,v,v_5)$-critical path if we recolor $vv_2$ with 5 for any
$i\ge 6$.

{\bf Claim 1.} $5\in F_{v_5}(\phi)$.

\begin{proof}
By contradiction, $5\notin F_{v_5}(\phi)$. Firstly, we show that
$\{3,4\}\subseteq F_{v_5}(\phi)$. Otherwise, assume that $3\notin
F_{v_5}(\phi)$. Since there is a $(3,i,v,v_5)$-critical path for any
$i\in \{6,7\}$ if we recolor $v_5x_5$ with 3,
$F_{v_3}(\phi)=\{3,6,7\}$. Exchange the colors on $vv_2$ and $vv_3$,
color $vv_5$ with 6. The resulting coloring is an acyclic edge
coloring of $G$, a contradiction.  Thus
$F_{v_5}(\phi)=\{2,3,4,6,7\}$. Note that there is a $(1,i,v,
v_5)$-critical path through $v_1$ for any $i\in \{6,7\}$, if we
recolor $v_5x_5$ with 1.

Secondly, we show that $\{1,2\}\subseteq F_{y_5}(\phi)$. Otherwise,
choose a color in $\{1,2\}$ not appearing on $y_5$ to recolor
$v_5y_5$, then recolor $v_5x_5$ with 5 and color $vv_5$ with 6. Thus
$F_{y_5}(\phi)=\{1,2,5,6,7\}$.

Let $\phi'(v_5y_5)=3$, $\phi'(v_5x_5)=5$, and $\phi'(e)=\phi(e)$ for
other edges of $G$. Then $\phi'$ is an acyclic edge coloring of
$G'$. If there is no $(3,i,v,v_5)$-critical path with respect to
$\phi'$ for some $i\in \{6,7\}$, then color $vv_5$ with $i$.
Otherwise, exchange the colors on $vv_2$ and $vv_3$, color $vv_5$
with 6.
\end{proof}

It follows from Claim 1 that $\{3,4\}\cap C_{x_5}(\phi)\neq
\emptyset$, say $3\in C_{x_5}(\phi)$. Note that there is a
$(5,3,v_5,x_5)$-critical path, since otherwise, recoloring $x_5v_5$
with 3 results in a new acyclic edge coloring $\phi'$ of $G'$. If
there is no $(3,i,v,v_5)$-critical path with respect to $\phi'$ for
some $i\in \{6,7\}$, then color $vv_5$ with $i$. Otherwise, exchange
the colors on $vv_2$ and $vv_3$, color $vv_5$ with 6.

{\bf Subcase 1.1.1} $4\in C_{x_5}(\phi)$.

By the same argument above, there is a $(5,4,v_5,x_5)$-critical
path. Recolor $x_5v_5$ with 3, $v_5y_5$ with 2, and color $vv_5$
with 5.

{\bf Subcase 1.1.2} $1\in C_{x_5}(\phi)$.

If $1\in C_{y_5}(\phi)$, then recolor $v_5y_5$ with 1 and color
$vv_5$ with 5. Otherwise, recolor $x_5v_5$ with 3, $v_5y_5$ with 2,
and color $vv_5$ with 5.

{\bf Subcase 1.2.} $\phi(v_5x_5)=1$.

For any $i\geq 6$, there is a $(1,i,v,v_5)$-critical path through
$v_1$ and $x_5$. Note that for any $i\in \{2,3,4\}\setminus
F_{x_5}(\phi)$, there is a $(5,i,v_5,x_5)$-critical path through
$y_5$, since otherwise, we can recolor $v_5x_5$ with $i$ and the
argument in Subcase 1.1 works. Thus $5\in F_{x_5}(\phi)$, without
loss of generality, let $d(x_5)=\Delta(G)$ and
$F_{x_5}(\phi)=\{4,5,\ldots,k\}$.


Color $vv_5$ with 2 and uncolor $vv_2$, we get an acyclic
edge coloring $\varphi$ of $G''=G-vv_2$. Then  $F_{v}(\varphi)\cap F_{v_2}(\varphi)\neq \emptyset$ by Fact \ref{uv-edge-degree}. We consider the
following two cases.

{\bf Subcase 1.2.1} $|F_{v}(\varphi)\cap F_{v_2}(\varphi)|=1$.

Assume that $\varphi(v_2x_2)\in F_{v}(\varphi)$. If $\varphi(v_2x_2)\in
\{3,4\}$, then the proof in Subcase 1.1 works. Otherwise,
$\varphi(v_2x_2)=1$. We color $vv_2$ with a color from $\{6,7\}\setminus F_{v_2}(\varphi)$.

{\bf Subcase 1.2.2} $|F_{v}(\varphi)\cap F_{v_2}(\varphi)|=2$.

Note that  $1\notin F_{v_2}(\varphi)$, since otherwise, assume that
$\varphi(v_2x_2)=1$ and $\varphi(v_2y_2)=i$, where $i\in \{3,4\}$.
Choose a color from $\{5,6,7\}\setminus F_{v_i}(\varphi)$ to color
$vv_2$. Thus
$F_{v_2}(\varphi)=\{3, 4\}$. For any $i\ge 5$, there is either a
$(3,i,v,v_2)$-critical path  through $v_3$, or a $(4,i,v,v_2)$-critical path through $v_4$. This implies that there exists a vertex in $\{v_3, v_4\}$,
say $v_3$, such that the colors in $S_{vv_3}(\varphi)$ are
greater than or equal to 5, say $S_{vv_3}(\varphi)=\{5,6\}$. Recolor $vv_3$ with 7 and choose a color from $\{5,6\}\setminus F_{v_4}(\varphi)$ to color $vv_2$, it is possible since $7\in F_{v_3}(\varphi)$.

{\bf Case 2.} $|F_{v}(\phi)\cap F_{v_5}(\phi)|=2$.

Then for any color $i$ with $5\le i\le k$, there is either a
$(\phi(v_5x_5), i, v, v_5)$-critical path, or  a $(\phi(v_5y_5),$ $
i, v, v_5)$-critical path. If there is a color $i_0\ge 5$ such that
$i_0$ not appearing on $x_5$ (or $y_5$), then we recolor $x_5v_5$
(or $y_5v_5$) with $i_0$ and the argument in Case 1 works.
Otherwise, $\{5,\ldots,k\}\subseteq F_{x_5}(\phi)\cap
F_{y_5}(\phi)$. In this case, if there is a vertex $v_i$ for some
$i\in \{2,3,4\}$ such that the colors in $F_{vv_i}(\phi)$ are
greater than or equal to 5, then recolor $vv_i$ with a color from
$\{4,5,\ldots,k\}\setminus F_{vv_i}(\phi)$ and the argument in Case
1 works. Otherwise, $1\in F_{v_5}(\phi)$. Assume that
$\phi(x_5v_5)=1$ and $\phi(y_5v_5)=2$. Since
$\{5,6,\ldots,k\}\subseteq F_{x_5}(\phi)$, $\{3,4\}\cap
C_{x_5}(\phi)\neq \emptyset$, say $3\in C_{x_5}(\phi)$. Then there
is a $(2,3,v_5,x_5)$-critical path through $y_5$, since otherwise,
we can recolor $x_5v_5$ with 3 and choose a color from
$\{4,5,\ldots,k\}\setminus (F_{v_2}(\phi)\cup F_{v_2}(\phi))$ to
color $vv_5$. This implies that $4\notin F_{x_5}(\phi)$. We can
recolor recolor $x_5v_5$ with 4 and choose a color from
$\{4,5,\ldots,k\}\setminus (F_{v_2}(\phi)\cup F_{v_4}(\phi))$ to
color $vv_5$.

In any case, we can get a acyclic edge coloring of $G$ with $\Delta(G)+2$ colors, a contradiction.
\end{proof}


\section{Discharging}

In this section, we give the proofs of Theorem \ref{mad<4} and
\ref{mad<3} by discharging method.

\subsection{Proof of Theorem \ref{mad<4}}

\begin{pf1} By contradiction, assume that $G$ is an acyclically edge $(\Delta(G)+2)$-critical graph, i.e., $G$ is the minimum counterexample to the theorem in terms of the number of edges. Then $G$ is 2-connected by Lemma \ref{2-connected}. It follows from Lemmas \ref{neighbor of 2-vertex} and \ref{3-vertex-adjacent-4-vertex} that $\Delta(G)\ge 4$. Let us assign an initial charge of $\omega(v)=d(v)-4$ to each vertex $v\in V(G)$. It follows from $\mad(G)<4$ that $$\sum\limits_{v\in{V(G)}}\omega(v)<0.$$ We shall design some discharging rules and redistribute weights according to them. Once the discharging is finished, a new weight function $\omega'$ is produced. However, the total sum of weights
is kept fixed when the discharging is in process. On the other hand, we shall show that $\omega'(v)\geq 0$ for all $v\in V(G)$. This leads to an obvious contradiction. The discharging rules are
defined as follows.

$(R_1)$ Every 2-vertex receives 1 from each neighbor.

$(R_2)$ Let $v$ be a 3-vertex. If $v$ is adjacent to a 4-vertex,
then the other neighbors of $v$ are $5^+$-vertices by Lemma
\ref{3-vertex-adjacent-4-vertex}, and $v$ receivers $\frac{1}{2}$
from each adjacent $5^+$-vertex. Otherwise, $v$ receivers
$\frac{1}{3}$ from each neighbor.

Now we show that the resultant charge function  $\omega'(v)\geq 0$ for
any $v\in V(G)$.

If $d(v)=2$, then $\omega(v)=-2$ and the neighbors of $v$ are $5^+$-vertices by Lemma \ref{neighbor of 2-vertex}, each of which gives 1 to $v$, so $\omega'(v)=-2+2\times1=0$.

Suppose that $d(v)=3$.  Then $\omega(v)=-1$ and the neighbor of $v$
are $4^+$-vertices by Lemma \ref{lem3vertex}. If $v$ is adjacent to
a 4-vertex, then the other neighbors of $v$ are $5^+$-vertices by
Lemma \ref{3-vertex-adjacent-4-vertex},  so the other neighbors of
$v$ are $5^+$-vertices by Lemma \ref{3-vertex-adjacent-4-vertex}
each of which gives $\frac{1}{2}$ to $v$, so
$\omega'(v)=-1+2\times\frac{1}{2}=0$. Otherwise,
$\omega'(v)=-1+3\times\frac{1}{3}=0$.

If $d(v)=4$, then $\omega'(v)=\omega(v)=0$.

Suppose that $d(v)\ge5$. Let $u$ be the neighbor of $v$ which has
minimum degree in $N(v)$. If $d(u)\ge 4$, then
$\omega'(v)=\omega(v)>0$. We first consider the case that $d(u)=2$,
then $n_2(v)+n_3(v)\le d(v)-3$ by Lemma \ref{lem2vertex}. If
$n_2(v)\leq d(v)-5$, then $\omega'(v)=\omega(v)-n_2(v)\times
1-n_3(v)\times\frac{1}{2}\ge
d(v)-4-n_2(v)-\frac{d(v)-3-n_2(v)}{2}=\frac{d(v)-n_2(v)-5}{2}\geq
0$. Otherwise, $n_2(v)=d(v)-4$ and $n_3(v)=0$ by Lemma
\ref{k-vertex-2-ver}. Thus $\omega'(v)=\omega(v)-n_2(v)\times 1=0$.
At last we consider the case that $d(u)=3$. If $u$ is adjacent to a
$4$-vertex, then $n_3(v)\le d(v)-2$ if $d(v)\ge6$, and $n_3(v)\le 2$
if $d(v)=5$ by Lemma \ref{3-vertex-adjacent-4-vertex}. Thus
$\omega'(v)\ge \omega(v)-2\times \frac{1}{2}=0$ if $d(v)=5$; and
$\omega'(v)\ge \omega(v)-(d(v)-2)\times \frac{1}{2}\ge 0$ if
$d(v)\ge 6$. Otherwise, $u$ receives $\frac{1}{3}$ from $v$. If
$d(v)=5$, then $n_3(v)\leq 3$ by Lemma \ref{5-vertex}, so
$\omega'(v)\ge \omega(v)-3\times \frac{1}{3}= 0$. Otherwise,
$\omega'(v)\ge \omega(v)-d(v)\times \frac{1}{3}\ge 0$.

Thus $\omega'(v)\geq 0$ for any $v\in V(G)$, a contradiction. This
completes the proof.
\end{pf1}




\subsection{Proof of Theorem \ref{mad<3}}

\begin{pf2}  By contradiction, assume that $G$ is an acyclically edge $(\Delta(G)+1)$-critical graph. Then $G$ is 2-connected by Lemma \ref{2-connected}. Let us assign an initial charge of $\omega(v)=d(v)-3$ to each vertex $v\in V(G)$. It follows from $\mad(G)<3$ that $$\sum\limits_{v\in{V(G)}}\omega(v)<0.$$  The discharging rule is
defined as follows.

$(R)$ Every 2-vertex receives $\frac{1}{2}$ from each neighbor.

Now we show that the resultant charge function  $\omega'(v)\geq 0$
for any $v\in V(G)$. If $d(v)=2$, then $\omega(v)=-1$ and the
neighbors of $v$ are $4^+$-vertices by Lemma \ref{neighbor of
2-vertex}, each of which gives $\frac{1}{2}$ to $v$, so
$\omega'(v)=-1+2\times\frac{1}{2}=0$. If $d(v)=3$, then
$\omega'(v)=\omega(v)=0$. Otherwise, $d(v)\ge 4$.  By Corollary
\ref{2-vertex-+1},  $n_2(v)\le d(v)-2$ and  $\omega'(v)\ge
\omega(v)-(d(v)-2)\times\frac{1}{2}=\frac{d(v)}{2}-2\ge0$.

Thus $\omega'(v)\geq 0$ for any $v\in V(G)$, a contradiction. This
completes the proof.
\end{pf2}




\section{Open Problems}

Since any cycle has acyclic chromatic index three, the bound $\Delta(G)+1$ in Theorem \ref{mad<3} is tight. It is easy to see that the graphs $K_4, K_{3,3}$ have acyclic chromatic index five. So the bounds $\mad<3$ in Theorem \ref{mad<3} and $\Delta(G)+2$ in Theorem \ref{mad<4} are tight.

Determining the acyclic chromatic index of a graph is a hard
problem both from theoretical and algorithmic points of view. Even
for complete graphs, the acyclic chromatic index is still not
determined exactly. It has been shown by Alon and Zaks
\cite{Alon2002} that determining whether $\chi'_a(G)\leq 3$ is
NP-complete for an arbitrary graph $G$.
Now we provide the following open problems.

\begin{problem} Find the necessary or sufficient conditions for a
graph  $G$ with $\chi'_a(G)=\chi'(G)$, where $\chi'(G)$ is the
chromatic index of $G$.
\end{problem}

\begin{problem} Determine  the acyclic  chromatic index of
planar graphs.
\end{problem}

\end{document}